\documentclass[11pt,reqno]{amsart}
\usepackage{amssymb,amsmath,amsthm}
\usepackage{comment,enumitem,url}
\usepackage{xcolor}
\usepackage{cleveref}
\usepackage{mathtools}
\usepackage{a4wide,fullpage}
\usepackage{bm}

\newtheorem{theorem}{Theorem}
\newtheorem{lemma}[theorem]{Lemma}
\newtheorem{corollary}[theorem]{Corollary}

\newtheorem{problem}[theorem]{Problem}

\newcounter{specialstatement}

\newtheorem{conjecture}[specialstatement]{Conjecture}

\theoremstyle{remark}

\theoremstyle{definition}
\newtheorem*{remark}{Remark}

\numberwithin{theorem}{section} \numberwithin{equation}{section}

\setlist[enumerate]{leftmargin=*,label=\rm{(\arabic*)}}
\setlist[itemize]{leftmargin=*}

\renewcommand{\l}{\lambda}

\renewcommand{\a}{\alpha}

\newcommand{\Z}{\mathbb{Z}}
\newcommand{\N}{\mathbb{N}}

\newcommand{\z}{\zeta}

\renewcommand{\pmod}[1]{\ \, \left( \mathrm{mod} \, #1 \right)}

\allowdisplaybreaks
\usepackage{xr}
\makeatletter
\newcommand*{\addFileDependency}[1]{
	\typeout{(#1)}
	\@addtofilelist{#1}
	\IfFileExists{#1}{}{\typeout{No file #1.}}
}
\makeatother

\makeatletter
\@namedef{subjclassname@2020}{%
	\textup{2020} Mathematics Subject Classification}
\makeatother

\title{Congruences for odd unimodal sequences and a problem of Bringmann and Lovejoy}
\author{Koustav Banerjee}
\address{University of Cologne, Department of Mathematics and Computer Science, Weyertal 86-90, 50931 Cologne, Germany}
\email{kbanerj1@uni-koeln.de}
\author{Manosij Ghosh Dastidar}
\address{Technische Universit\"at Wien, Wiedner Hauptstrasse 8--10/104, 1040 Wien, Austria}
\email{gdmanosij@gmail.com}

\allowdisplaybreaks


\begin{document}
\maketitle

\begin{abstract}
In this paper, we give a family of the arithmetic progressions in which the numbers of odd unimodal and odd strongly unimodal sequences are divisible by $2$ and by $4$. In particular, we answer a question of Bringmann and Lovejoy on the existence of such arithmetic progressions for odd strongly unimodal sequences modulo $4$. 
\end{abstract}

\section{Introduction and statement of main results}
A \emph{partition} $\l$ of $n\in \mathbb{N}_0$ is a weakly decreasing sequence of positive integers
$(\lambda_1,\cdots,\lambda_{\ell})$ with $\sum_{j=1}^r\l_j=n$ and $p(n)$ denotes the
number of partitions of $n$ with $p(0):=1$. Euler showed that
\begin{equation*}
P(q):=\sum_{n\ge0}p(n)q^n=\frac1{(q)_\infty}.
\end{equation*}
Here, for $n\in \mathbb{N}_0\cup \{\infty\}$, $(a)_n=(a;q)_n:=\prod_{j=0}^{n-1}(1-aq^j)$ denotes the usual {\it $q$-Pochhammer symbol}. Ramanujan \cite{Ramanujan1,Ramanujan2} discovered the congruences
\[
p(5n+4)\equiv0\pmod5,\qquad p(7n+5)\equiv0\pmod7,\qquad p(11n+6)\equiv0\pmod{11}
\]
and conjectured infinite families of congruences modulo powers of $5$, $7$ and $11$. Watson settled \cite{Watson} the conjectures for powers of $5$ and $7$, and later Atkin \cite{Atkin} settled the case for powers of $11$. Using the theory of modular
forms, Ono \cite{Ono2} proved that for $p\ge5$ prime, there exist infinitely many
arithmetic progressions $An+B$ with $p(An+B)\equiv0\pmod p$. Ahlgren \cite{Ahlgren} further extended this family to all moduli coprime to $6$. But, determining the parity of $p(n)$ remains a hard problem. Parkin and Shanks \cite{PS} conjectured that the parity of $p(n)$ is equidistributed. In this aspect, Subbarao~\cite{Subbarao} conjectured that every arithmetic progression contains infinitely
many integers $n$ for which $p(n)$ is even, and infinitely many for which $p(n)$ is odd. Ono \cite{Ono1}
settled the even case, and showed that the odd case holds for a progression if there exists $n$ with $p(n)$ odd. Radu \cite{Radu} settled the Subbarao's
conjecture for the odd case. Recently, Ono and Swaminathan \cite{OS} proved that $p(n)$ takes both parities infinitely often along the quadratic progressions $n=\frac{Dm^2+1}{24}$ with $\gcd(m,6)=1$, for all square-free $D>1$ with $D\equiv-1\pmod{24}$. 

In proving Ramanujan's congruences, modularity\footnote{Note that, $P$ is basically a weakly holomorphic modular form of weight $-\frac 12$ on ${\rm SL}_2(\mathbb{Z})$.} of $P$ plays a crucial role. But many
natural combinatorial generating functions are not modular, instead are built from modular forms
together with false theta functions or mock theta functions. One such combinatorial object is unimodal sequence (see \cite{Auluck,Wright}). A sequence of positive integers is called {\it unimodal} if
\begin{equation}\label{eq:unimodal}
a_1\le a_2\le\dots\le a_r\le c\ge b_1\ge b_2\ge\dots\ge b_s\quad \text{with}\ \ \sum_{j=1}^{r}a_j+c+\sum_{j=1}^{s}b_j=n\quad  \left(r,s\in \mathbb{N}_0\right).
\end{equation}
Let $u(n)$ be the number of unimodal sequences of size $n$, with $u(0):=1$. The corresponding generating function is given by
\[
U(q):=\sum_{n\ge0}u(n)q^n=\frac1{(q)_\infty^2}\sum_{n\ge0}(-1)^nq^{\frac{n(n+1)}2}
\]
We note that $U$ is a {\it mixed false theta function}, i.e., combination of a modular form and a false theta function. A sequence is {\it strongly unimodal} if inequalities in \eqref{eq:unimodal} are strict and the number of such sequences is denoted by ${\rm u}^*(n)$. The corresponding generating function is given by
\[
U^*(q):=\sum_{n\ge0}u^*(n)q^n=\sum_{n\ge0}(-q)^2_nq^{n+1}.
\]
Bryson, Ono, Pitman, and Rhoades \cite{BOPR} showed that $U^*$ is related to mock and quantum modular
forms. Congruences modulo $2$ and $4$ for unimodal-type sequences were conjectured
by \cite{BOPR} and subsequently settled by Chen and Garvan \cite{CG}.

Bringmann and Lovejoy \cite{BL} introduced two restricted unimodal sequences by imposing the parity of its parts. An {\it odd} unimodal sequence is a unimodal sequences in which all parts are odd. Let ${\rm ou}(n)$ and ${\rm ou}^*(n)$ be the numbers of odd unimodal and odd strongly
unimodal sequences of size $n$, with ${\rm ou}(0)={\rm ou}^*(0):=0$. Due to Bringmann and Lovejoy \cite[Theorems~1.1 and~1.2]{BL}, we have,\footnote{Setting $\z=1$ in Theorems 1.1 and 1.2 of \cite{BL}, we get \eqref{eq:gf}.}
\begin{equation}\label{eq:gf}
{\rm OU}(q):=\sum_{n\ge0}{\rm ou}(n)q^n=\sum_{n\ge0}\frac{q^{2n+1}}{(q;q^2)_{n+1}^2},\qquad
{\rm OU}^*(q):=\sum_{n\ge0}{\rm ou}^*(n)q^n=\sum_{n\ge0}q^{2n+1}(-q;q^2)_n^2.
\end{equation}
Up to simple terms, ${\rm OU}$ is a mixed false
theta function and ${\rm OU}^*$ is a mixed mock theta function.
Bringmann and Lovejoy \cite[Proposition 6.1 and Theorem 6.2]{BL} proved the following congruences for ${\rm ou}^*(n)$. 
\begin{theorem}\label{thm:BLparity}
	For $n\in \mathbb{N}$, ${\rm ou}^*(n)$ is odd if and only if $6n-2$ is a square.
\end{theorem}
\begin{theorem}\label{thm:BL62}
	Let $k\in \mathbb{N}_0$ and $r\in \mathbb{N}$ with $1\le r\le k+1$, let $p_r\ge 5$ be prime. For any $j\not\equiv 0\pmod{p_{k+1}}$, if $p_{k+1}\not\equiv 7, 13 \pmod{24}$ or $p_{k+1}\equiv 7, 13\pmod{24}$ and $(\frac{3j}{p_{k+1}})=-1$, then with $P:=\prod_{r=1}^{k}p_r$, we have,
	\begin{align*}
	{\rm ou}^*\left(4P^2p^2_{k+1}n+2P^2p_{k+1} j+\frac{8P^2p^2_{k+1}+1}{3}\right)&\equiv0\pmod4\qquad(j\ \text{odd}),\\
	{\rm ou}^*\left(4P^2p^2_{k+1}n+2P^2p_{k+1} j+\frac{2P^2p^2_{k+1}+1}{3}\right)&\equiv0\pmod4\qquad(j\ \text{even}).
	\end{align*}
\end{theorem}
In particular, Bringmann and Lovejoy \cite[(7.1)--(7.3)]{BL} showed that for $k=0$ and $p\in \{5,7,11\}$, \Cref{thm:BL62} gives
\begin{align}\label{eqn1}
{\rm ou}^*(100n+r)&\equiv0\pmod4,\quad r\in\{37,57,77,97\},\nonumber\\
{\rm ou}^*(196n+r)&\equiv0\pmod4,\quad r\in\{61,89,145\},\nonumber\\
{\rm ou}^*(484n+r)&\equiv0\pmod4,\quad r\in\{125,169,213,257,301,345,389,433,477,521\}.
\end{align}
Moreover, by numerical verifications, they found that the cases $r\in\{37,97\}$ of
\eqref{eqn1} appear to be special cases of ${\rm ou}^*(50n+r)\equiv0\pmod4$ for $r\in\{37,47\}$,
the cases $r\in\{61,145\}$ appear to be special cases of ${\rm ou}^*(98n+r)\equiv0\pmod4$ for $r\in\{47,61\}$, and that all congruences in \eqref{eqn1} appear to be special cases of congruences modulo $242$. Based on these observations, they posed \cite[Section 7 (2)]{BL} the following open problem.
\begin{problem}\label{prob:BL}
	Determine exactly which of the congruences of \Cref{thm:BL62} {\rm (or \cite[Theorem 1.5]{BL})} can be strengthened to the corresponding progression with modulus $2p^2$.
\end{problem}

We first state results for congruences of ${\rm ou}(n)$ modulo $2$ and $4$.
\begin{theorem}\label{thm1}
\begin{enumerate}[label=\textup{(\roman*)}]
	\item For $n\in \mathbb{N}$, ${\rm ou}(n)$ is odd iff $3n+1=4m^2$ for some $m\in\Z$.
	\item Let $n\in \mathbb{N}$ be odd. Then ${\rm ou}(n)\equiv0\pmod4$ iff $3n+1$ is not a perfect square.
	\item Let $n\in \mathbb{N}$ be even. If $3n+1\ne m^2_1+3m^2_2$ for any $m_1,m_2\in \mathbb{Z}$, then ${\rm ou}(n)\equiv 0\pmod{4}$. 
\end{enumerate}
\end{theorem}

Next, we state a classification for congruences of ${\rm ou}^*(n)$ modulo $2$ and $4$. Define, for $\varepsilon\in \{1,2\}$,
\begin{align}\label{setdef}
S_1&:=\left\{2m^2:m\in \mathbb{N}\right\},\nonumber\\
S^{[\varepsilon]}_2&:=\left\{\varepsilon p^{4k+1}m^2:\left(\frac{-2}{p}\right)=-1, k\in \mathbb{N}_0,\ p\nmid m,\ \text{and}\ m\in \mathbb{N} \right\},\quad  S_2:=S^{[1]}_2\cup S^{[2]}_2.
\end{align} 
\begin{theorem}\label{thm2}
	Let $n\ge1$ be odd. We have,
	\[
	{\rm ou}^*(n)\equiv\begin{cases}1 \pmod2&\Leftrightarrow\ 3n-1\in S_1,\\
	2\pmod4&\Leftrightarrow\ 3n-1\in S_2,\\
	0\pmod4&\Leftrightarrow\ 3n-1\notin S_1\cup S_2.\end{cases}
	\]
\end{theorem}

As a direct consequence of \Cref{thm2}, we have the following:
\begin{corollary}\label{cor1}
Let $A,B\in \mathbb{N}_0$ with $A$ even and $B$ odd. Then, for $n\in \mathbb{N}_0$, we have
\[
{\rm ou}^*\left(An+B\right)\equiv 0\pmod{4}\Leftrightarrow 3(An+B)-1\notin S_1\cup S_2.
\]
\end{corollary}

For $M\in\mathbb N$ with $\gcd(M,6)=1$, a prime $\ell\ge5$ and $j\in\mathbb N$ with
$\ell\nmid j$, set
\[
R_{M,\ell,j}:=2M^2\ell j+\frac{2M^2\ell^2+1}{3}.
\]
\begin{corollary}\label{C:GD2}
	Let $M\in\mathbb N$ with $\gcd(M,6)=1$, $\ell\ge5$ be a prime, and $j\in\mathbb{N}$ with $\ell\nmid j$. Assume that $\nu_\ell(M)$ is odd or that one of the following holds:
	\begin{enumerate}[label=\textnormal{(\arabic*)}]
		\item $\ell\equiv1,11,17,19\pmod{24}$;
		\item $\ell\equiv5,23\pmod{24}$ and $\left(\frac j\ell\right)=-1$;
		\item $\ell\equiv7,13\pmod{24}$ and $\left(\frac{3j}\ell\right)=-1$.
	\end{enumerate}
	Then, for all $n\in\mathbb{N}_0$, we have
	\[
	\mathrm{ou}^*\left(2M^2\ell^2n+R_{M,\ell,j}\right)\equiv0\pmod 4.
	\]
Otherwise, for infinitely many $n\in\mathbb{N}_0$,
\[
\mathrm{ou}^*\left(2M^2\ell^2n+R_{M,\ell,j}\right)\equiv2\pmod 4.
\]
\end{corollary}

With $M=1$, we have the following result. 
\begin{corollary}\label{C:GD1}
	Let $\ell\ge5$ be prime and $j\in\mathbb N$ with $\ell\nmid j$, and assume that one of the conditions
	\textnormal{(1)--(3)} of \Cref{C:GD2} holds. Then, for $n\in\mathbb N_0$, we have
	\[
	\mathrm{ou}^*\left(2\ell^2n+R_{1,\ell,j}\right)\equiv0\pmod 4.
	\]
	Otherwise, for infinitely many $n\in\mathbb{N}_0$,
	\[
	\mathrm{ou}^*\left(2\ell^2n+R_{1,\ell,j}\right)\equiv2\pmod 4.
	\]
\end{corollary}

\begin{remark}
	\Cref{C:GD1} settles \Cref{prob:BL}.
\end{remark}

The remainder of the paper is organized as follows. In \Cref{sec1}, we recall well known $q$-series identities and congruences related to modular forms and mock theta functions. In \Cref{sec2}, we prove \Cref{thm1}. \Cref{sec3} presents proofs of \Cref{thm2}, \Cref{C:GD2}, and \Cref{C:GD1}. Finally, we conclude this paper with problems for future research in \Cref{sec4}.

\section*{Acknowledgements}
The first author has received funding from the European Research Council (ERC) under the European Unions Horizon 2020 research and innovation programme (grant agreement No. 101001179). The second author was supported by Austrian Science Fund (FWF): P 34142.

\section{Preliminaries}\label{sec1}

First, from \cite[(4.1)]{BL}, we have the following representation of ${\rm OU}(q)$ as a linear combination of mixed false theta function.
\begin{lemma}\label{lem1}
We have,
\[
{\rm OU}(q)=\sum_{n\ge 0}(-1)^{n+1}q^{3n^2+2n}\left(1+q^{2n+1}\right)+\frac{1}{\left(q;q^2\right)^2_{\infty}}\sum_{n\ge 0}(-1)^n q^{n(n+1)}.
\]
\end{lemma}
By \cite[Theorem 1.60]{Ono} and \cite[(1.3.1)]{Andrews}, we have the following identities: 
\begin{lemma}\label{lem2}
We have,
\[
\sum_{n\ge 0}q^{(2n+1)^2}=q\frac{\left(q^{16};q^{16}\right)^2_{\infty}}{\left(q^8;q^8\right)_{\infty}},\qquad \sum_{n\in \mathbb{Z}}(-1)^n q^{n^2}=\frac{(q)^2_{\infty}}{\left(q^2;q^2\right)_{\infty}},\quad (q)_{\infty}=\sum_{n\in \mathbb{Z}}(-1)^n q^{\frac{n(3n-1)}{2}}.
\]
\end{lemma}

By \cite[Theorem 3.7.7]{Berndt}, we obtain the following twisted divisor sum representation of the eta quotient $\frac{(q^2;q^2)^2_{\infty}}{(q)_{\infty}}$. Define
\[
D_3(n):=\sum_{d\mid n}\chi_{-3}\left(d\right),
\]
where $\chi_{k}(d):=(\frac{k}{d})$ denotes the {\it Kronecker symbol}.
\begin{lemma}\label{lem4}
We have,
\[
\frac{\left(q^2;q^2\right)^2_{\infty}\left(q^6;q^6\right)^2_{\infty}}{(q)_{\infty}\left(q^3;q^3\right)_{\infty}}=\sum_{n\ge 0}D_3\left(2n+1\right)q^n.
\]
\end{lemma}

Next, we recall Ramanujan's third order mock theta functions (see \cite{AB5}).
\[
f(q):=\sum_{n\ge0}\frac{q^{n^2}}{(-q)_n^2},\quad \omega(q):=\sum_{n\ge0}\frac{q^{2n(n+1)}}{\left(q;q^2\right)_{n+1}^2},\quad \nu(q):=\sum_{n\ge0}(q;q^2)_n(-q)^n,\quad \psi(q)=\sum_{n\ge1}\frac{q^{n^2}}{\left(q;q^2\right)_n}.
\]
Due to Watson \cite{W} and Fine \cite[(26.88)]{Fine}, we have the following identities for mock theta functions.
\begin{lemma}\label{thm:watson}
We have,
	\begin{align}
	f(q^8)+2q\omega(q)+2q^3\omega\left(-q^4\right)&=\frac{\left(q^2;q^2\right)_\infty\left(q^4;q^4\right)_\infty^6}{(q)_\infty^2\left(q^8;q^8\right)_\infty^4},\label{W1}\\
	f(q)+4\psi(-q)&=\frac{(q)_\infty^3}{\left(q^2;q^2\right)_\infty^2},\label{W2}\\
	\nu(-q)&=q\omega\left(q^2\right)+\frac{\left(q^4;q^4\right)_\infty^3}{\left(q^2;q^2\right)_\infty^2}.\label{thm:fine}
	\end{align}
\end{lemma}

Moreover, we require the following $4$-dissection from \cite[Lemma 3.1]{APSY}.
\begin{lemma}\label{lem:phi4a}
	We have,
	\[
	\frac{\left(q^2;q^2\right)_\infty\left(q^4;q^4\right)_\infty^6}{(q)_\infty^2\left(q^8;q^8\right)_\infty^4}=\sum_{j=0}^3 q^jF_j\left(q^4\right),
	\]
	where
	\begin{align*}
	F_0(q)&:=\frac{\left(q^2;q^2\right)_\infty^{15}}{(q)_\infty^8\left(q^4;q^4\right)_\infty^6},\quad F_1(q):=2\frac{\left(q^2;q^2\right)_\infty^{9}}{(q)_\infty^6\left(q^4;q^4\right)_\infty^2},\quad
	F_2(q):=4\frac{\left(q^2;q^2\right)_\infty^{3}\left(q^4;q^4\right)_\infty^2}{(q)_\infty^4},\\
	F_3(q)&:=8\,\frac{\left(q^4;q^4\right)_\infty^{6}}{(q)_\infty^2\left(q^2;q^2\right)_\infty^3}.
	\end{align*}
\end{lemma}

%Now, we record two identities regarding eta quotients from \cite[(5)--(8)]{APSY}.
%\begin{lemma}\label{lem:phi4}
%	We have,
%	\begin{align}
%	\frac{\left(q^2;q^2\right)_\infty^5}{(q)_\infty^2\left(q^4;q^4\right)_\infty^2}&=\frac{\left(q^8;q^8\right%)_\infty^5}{\left(q^4;q^4\right)_\infty^2\left(q^{16};q^{16}\right)_\infty^2}
%	+2q\frac{\left(q^{16};q^{16}\right)_\infty^2}{\left(q^8;q^8\right)_\infty},\label{E1}\\
%	\frac{\left(q^2;q^2\right)_\infty^{10}}{(q)_\infty^4\left(q^4;q^4\right)_\infty^4}&=\frac{\left(q^4;q^4\right)_\infty^{10}}{\left(q^2;q^2\right)_\infty^4\left(q^8;q^8\right)_\infty^4}
%	+4q\frac{\left(q^8;q^8\right)_\infty^4}{\left(q^4;q^4\right)_\infty^2}.\label{E2}
%	\end{align}
%\end{lemma}

Next, we require an elementary fact from \cite[Lemma 6.3]{BL}.
\begin{lemma}\label{lem3}
For $n\in \mathbb{N}$, we have that modulo $4$,
\[
\left(-q;q^2\right)^2_n-\left(-q^2;q^4\right)_n
\]
is an odd polynomial.
\end{lemma}

For $n\in \mathbb{N}$, let $r_2(n)$ denote the number of representations of $n$ by the quadratic form $n^2_1+n^2_2$. From \cite[Theorem 3.2.1]{Berndt}, we obtain the following lemma.
\begin{lemma}\label{lem5}
We have, for $n\in \mathbb{N}$,
\[
r_2(n)=4\sum_{d\mid n}\chi_{-4}(d).
\]
\end{lemma}

Define, for $n\in \mathbb{N}_0$, $c_\psi(n):=[q^n](\psi(q))$. Due to Wang \cite[Theorem~4.2]{Wa}, parity of the Fourier coefficients of $\psi(q)$ reads as follows:
\begin{lemma}\label{thm:d}
	Let $n\in \mathbb{N}$. We have
	\[
	c_{\psi}(n)\ \text{is odd}\Leftrightarrow 24n-1=p^{4k+1}m^2\ \text{with}\ k\in \mathbb{N}_0\ \text{and for some prime}\ p,\ p\nmid m,\ m\in \mathbb{Z}.
	\]
\end{lemma}

Define, for $n\in \mathbb{N}_0$, $c_1(n):=[q^n](\frac{(q^2;q^2)_\infty^3}{(q)_\infty^2})$. Moreover, we require the following classification of $c_1(n)$ modulo $4$ due to Chen and Chen
\cite[Theorems 3.6 and 4.2]{CC}.
\begin{lemma}\label{L:CC}
	Let $n\in\mathbb N_0$. We have,
		 \[c_1(n)\equiv2\pmod 4\Leftrightarrow 6n+1\in S_2^{[1]}.\]
\end{lemma}

Finally, we require the congruences for the Fourier coefficients of $q\omega(q)$ from \cite[Theorem 3.4]{APSY}. Define, for $n\in \mathbb{N}_0$, $c_\omega(n):=[q^n](q\omega(q))$.
\begin{lemma}\label{thm:apsy}
	For $n\in \mathbb{N}_0$, we have
	\[
	c_\omega\left(8n+r\right)\equiv 0\pmod{4}\ \ \text{with}\ r\in \{4,6\}.
	\]
\end{lemma}

\section{Proof of \Cref{thm1}}\label{sec2}

\begin{proof}[Proof of \Cref{thm1}]
By \Cref{lem1}, we have
\begin{align*}
\sum_{n\ge 0}{\rm ou}(n)q^n&\equiv \sum_{n\ge 0}q^{3n^2+2n}\left(1+q^{2n+1}\right)+\frac{\left(q^2;q^2\right)^2_{\infty}}{\left(q\right)^2_{\infty}}\sum_{n\ge 0} q^{n(n+1)}\\
&\equiv \sum_{n\in \mathbb{Z}}q^{3n^2+2n}+\left(q^8;q^8\right)_{\infty}\equiv \sum_{n\in \mathbb{Z}}q^{3n^2+2n}+\sum_{n\in \mathbb{Z}}q^{12n^2-4n}\pmod{2}.
\end{align*}
using \Cref{lem2}, $(q)^2_{\infty}\equiv (q^2;q^2)_{\infty}\pmod{2}$, and $(q^4;q^4)^2_{\infty}\equiv (q^8;q^8)_{\infty}\pmod{2}$ in the second step, \Cref{lem2} in the final step. Letting $q\mapsto q^3$ and multiplying both sides by $q$, it follows that,
\begin{align*}
\sum_{n\ge 0}{\rm ou}(n)q^{3n+1}&\equiv \sum_{n\in \mathbb{Z}} q^{(3n+1)^2}+\sum_{n\in \mathbb{Z}} q^{(6n-1)^2}\equiv \sum_{\substack{n\ge 1\\n\ \text{even}\\3\nmid n}}q^{n^2}\pmod{2}.
\end{align*} 
This concludes the proof of \Cref{thm1} (i).

Next, by \Cref{lem1}, we have
\begin{align}\label{eqn2}
&\sum_{n\ge 0}{\rm ou}(n)q^n=\sum_{n\ge 0}(-1)^{n+1}q^{3n^2+2n}\left(1+q^{2n+1}\right)+\frac{\left(q^2;q^2\right)^2_{\infty}}{\left(q\right)^2_{\infty}}\sum_{n\ge 0} (-1)^nq^{n(n+1)}\nonumber\\
&\hspace{0.5 cm}\equiv \sum_{n\ge 0}(-1)^{n+1}q^{3n^2+2n}\left(1+q^{2n+1}\right)+(q)^2_{\infty}\sum_{n\ge 0} (-1)^nq^{n(n+1)}\nonumber\\
&\hspace{0.5 cm}\equiv \sum_{n\ge 0}(-1)^{n+1}q^{3n^2+2n}\left(1+q^{2n+1}\right)+\frac{(q)^2_{\infty}}{\left(q^2;q^2\right)_{\infty}}\sum_{\substack{n\ge 0\\m\in \mathbb{Z}}} (-1)^{m+n} q^{n(n+1)+m\left(3m-1\right)}\pmod{4},
\end{align}
using $(q)^4_{\infty}\equiv (q^2;q^2)^2_{\infty}\pmod{4}$ in the second step and \Cref{lem2} in the final step. For $N\in \mathbb{N}_0$ and\footnote{Here and throughout, $R[[q]]$ denotes the ring of formal power series in $q$ over a commutative ring $R$.} $A(q):=\sum_{n\ge 0}a(n)q^n\in \mathbb{Z}[[q]]$, $[q^N](A(q)):=a(N)$. We set 
\begin{equation}\label{eqn2a}
S_1(q):=\sum_{n\ge 0}(-1)^{n+1}q^{3n^2+2n}\left(1+q^{2n+1}\right)=:\sum_{n\ge 0}s_1(n)q^n
\end{equation}
and note that $s_1(n)\in \{-1,0,1\}$ for $n\in \mathbb{N}_0$. Moreover, since 
\begin{equation}\label{eqn2b}
3\left(3n^2+2n\right)+1=(3n+1)^2,\quad 3\left(3n^2+4n+1\right)+1=\left(3n+2\right)^2,
\end{equation}
we have, for $N$ odd and $3N+1\neq m^2$ (for all $m\in 2\mathbb{Z}$), 
\[
s_1(N)=0.
\]
Thus, with
\begin{equation}\label{eqn2c}
S_2(q):=\frac{(q)^2_{\infty}}{\left(q^2;q^2\right)_{\infty}},\quad S_3(q):=\sum_{\substack{n\ge 0\\m\in \mathbb{Z}}} (-1)^{m+n} q^{n(n+1)+m\left(3m-1\right)}=:\sum_{k\ge 0}s_3(k)q^k,
\end{equation} 
and using \Cref{lem2}, \eqref{eqn2} becomes
\begin{align*}
{\rm ou}\left(N\right)&\equiv \left[q^N\right]\left(S_2(q)S_3(q)\right)\equiv \sum_{m\in \mathbb{Z}}(-1)^m s_3\left(N-m^2\right)\equiv 2\sum_{\substack{m\ge 1\\2\nmid m}} s_3\left(N-m^2\right)\pmod{4},
\end{align*}
using $s_3(n)=0$ for $n$ odd in the final step. Thus, to conclude the proof, it suffices to show that
\begin{equation}\label{eqn4}
\left|X_1(N)\right|\equiv 0\pmod{2},\ \ \text{where}\ \ X_1(N):=\left\{m\in \mathbb{N}: 2\nmid m, s_3\left(N-m^2\right)\equiv 1\pmod{2}\right\}.
\end{equation}

Next, we show that
\begin{equation}\label{eqn5}
s_3(k)\equiv 1\pmod{2}\Leftrightarrow 3k+1=\ell^2\ \text{with}\ \gcd\left(\ell,6\right)=1.
\end{equation}
 Now, by definition of $s_3(k)$, we have
\begin{align}\label{eqn5a}
s_3(k)=\sum_{(n,m)\in X_2(k)}(-1)^{n+m}\ \text{with}\ X_2(k):=\left\{(n,m)\in \mathbb{N}_0\times \mathbb{Z}: n(n+1)+m(3m-1)=k\right\}.
\end{align}
Define $X_3(k):=\{(n,m)\in \mathbb{N}^2: 2\nmid n,\ \gcd(m,6)=1,\ m^2+3n^2=k\}$ and  $f_1:X_2(k)\rightarrow X_3(12k+4)$ with $f_1(n,m):=(2n+1,|6m-1|)$. Clearly, $f_1$ is well defined and bijective. Thus $|X_2(k)|=|X_3(12k+4)|$ and so
\begin{equation}\label{eqn6}
s_3(k)\equiv 1\pmod{2}\Leftrightarrow \left|X_3\left(12k+4\right)\right|\equiv 1\pmod{2}.
\end{equation}
By \Cref{lem4}, we have, for $3\nmid k$,
\begin{equation}\label{eqn7}
\left|X_3(k)\right|=\begin{cases}
D_3\left(\frac k4\right) &\text{if}\ k\equiv 4\pmod{8},\\
0 &\text{otherwise}.
\end{cases}
\end{equation}
Since, by definition, $D_3$ is multiplicative, for $k\in \mathbb{N}$ with $3\nmid k$, 
\begin{equation}\label{eqn7a}
D_3(k)\equiv 1\pmod{2}\Leftrightarrow k=\ell^2.
\end{equation}
Thus, combining \eqref{eqn6} and \eqref{eqn7}, it follows that
\[
s_3(k)\equiv 1\pmod{2}\Leftrightarrow 3k+1=\ell^2\ \text{with}\ \gcd\left(\ell,6\right)=1.
\]
Since for given $N$ and $m$, such $\ell$ is unique, thus, from \eqref{eqn4}, we get
\[
X_1(N)=\left\{m\in \mathbb{N}: 2\nmid m, \ell^2+3m^2=3N+1\ \text{with}\ \gcd\left(\ell,6\right)=1\right\}.
\]
Define $f_2:X_1(N)\rightarrow X_3(3N+1)$ by $f_2(m):=(m,\sqrt{3N+1-3m^2})$ and we observe that $f_2$ is a bijection. Thus,
\[
\left|X_1(N)\right|=\left|X_3\left(3N+1\right)\right|\equiv 0\pmod{2},
\]
using \eqref{eqn7}, \eqref{eqn7a}, and $3N+1\neq m^2$ for any $m$ even. 

Finally, assume that $3n+1$ is a square. Since $n$ is odd, $3n+1=4m^2$ for some $m\in \mathbb{Z}$ and by (i), ${\rm ou}(n)\not\equiv 0\pmod{4}$. This concludes the proof of \eqref{eqn4}.

Finally, we now prove (iii). By \eqref{eqn2a} and \eqref{eqn2b}, for $N\in \mathbb{N}$ even and $3N+1\neq m^2$ (since, by hypothesis $3N+1\neq m^2_1+3m^2_2$ for any $m_1,m_2\in \mathbb{Z}$), $s_1(N)=0$. Thus, with \eqref{eqn2c}, for $N$ even, \eqref{eqn2} becomes
\begin{align}\label{eqn8}
{\rm ou}(N)&\equiv \sum_{m\in \mathbb{Z}}(-1)^m s_3\left(N-m^2\right)=s_3(N)+2\sum_{m\ge 1}s_3\left(N-4m^2\right)\pmod{4},
\end{align}
using $s_3(n)=0$ for $n$ odd. Moreover, by \eqref{eqn5}, we obtain
\[
s_3\left(N-4m^2\right)\equiv 1\pmod{2}\Leftrightarrow 3N+1=\ell^2+12m^2
\]
and this contradicts the assumption that $3N+1=m^2_1+3m^2_2$ for any $m_1,m_2\in \mathbb{Z}$. Consequently, \eqref{eqn8} yields
\begin{equation}\label{eqn9}
{\rm ou}(N)\equiv s_3(N)\pmod{4}.
\end{equation} 
Since $N$ even and $3N+1\neq m^2_1+3m^2_2$ for any $m_1,m_2\in \mathbb{N}$, $X_2(N)=\emptyset$. Therefore, by \eqref{eqn5a}, $s_3(N)=0$ and so, from \eqref{eqn9}, we conclude the proof.\qedhere 
\end{proof}

\section{Proof of \Cref{thm2}, \Cref{C:GD2}, and \Cref{C:GD1}}\label{sec3}
To prove \Cref{thm2}, we require the following lemmas. Define $c_2(n):=[q^n]((q)_\infty(q^4;q^4)_\infty)$ (for $n\in \mathbb{N}_0$).
\begin{lemma}\label{assistlem0}
We have, for $n\in \mathbb{N}$ and modulo $4$, 
\begin{align*}
c_\omega(4n-3)&\equiv c_1(n-1),\quad c_\omega\left(8n-6\right)\equiv 2c_2(n-1),\quad c_\omega\left(8n-8\right)\equiv 2c_{\psi}(n-1),\\
 c_\omega(4n-1)&\equiv(-1)^n c_\omega(n).
\end{align*}
\end{lemma}
\begin{proof}
By \eqref{W1} and \Cref{lem:phi4a}, we have
\[
\sum_{n\ge 0}c_\omega\left(4n+1\right)q^n=\frac{\left(q^2;q^2\right)_\infty^{9}}{(q)_\infty^6\left(q^4;q^4\right)_\infty^2}\equiv \frac{\left(q^2;q^2\right)^3_{\infty}}{\left(q\right)^2_{\infty}}\pmod{4}.
\]
This concludes the proof of the first claim. 

Next, again using \eqref{W1} and \Cref{lem:phi4a}, we obtain
\[
\sum_{n\ge 0}c_\omega\left(4n+2\right)q^n=2\frac{\left(q^2;q^2\right)^3_\infty\left(q^4;q^4\right)^2_\infty}{(q)^4_\infty}\equiv 2\left(q^2;q^2\right)_\infty\left(q^8;q^8\right)_\infty\pmod{4}.
\]
Letting $n\mapsto 2n$, the second claim follows.

Next, using the fact $(q)_\infty^{8}\equiv(q^2;q^2)_\infty^{4}\pmod 8$ and the definition of $F_0$ (see \Cref{lem:phi4a}), we get, for $n\in \mathbb{N}_0$, 
\[
F_0(q)\equiv\frac{\left(q^2;q^2\right)_\infty^{11}}{\left(q^4;q^4\right)_\infty^6}\pmod 8.
\]
Thus, by \eqref{W1} and \Cref{lem:phi4a}, it follows that, for $n\in \mathbb{N}_0$,
\begin{align}\label{eqn10}
c_f(n)+2c_\omega\left(8n\right)=\left[q^{2n}\right]\left(F_0(q)\right)\equiv \left[q^{n}\right]\left(\frac{(q)_\infty^3}{\left(q^2;q^2\right)_\infty^2}\right)\pmod 8.
\end{align}
By \eqref{W2}, we have, for $n\in \mathbb{N}_0$,
\[
c_f(n)+4(-1)^nc_\psi(n)=\left[q^{n}\right]\left(\frac{(q)_\infty^3}{\left(q^2;q^2\right)_\infty^2}\right).
\]
Combining this with \eqref{eqn10} and then multiplying by $\frac 12$, we conclude the proof of the third claim.

Finally, from \eqref{W1}, \Cref{lem:phi4a}, we get, $n\in \mathbb{N}_0$,
\[
2c_{\omega}\left(4n+3\right)+2(-1)^nc_\omega(n+1)=\left[q^{n}\right]\left(F_3(q)\right)\equiv 0\pmod{8}.\qedhere 
\]
\end{proof}

Now, for $n$ odd, we relate ${\rm ou}^*(n)$ to $c_\omega(n)$ modulo $4$.

\begin{lemma}\label{assistlem2}
	For $n\in\N$ odd, we have
	\[
	{\rm ou}^*(n)\equiv \pm c_\omega(n)\pmod4.
	\]
\end{lemma}
\begin{proof}
	By \Cref{lem3}, we have that modulo $4$, for $n\in \mathbb{N}$,
	\[
	q^{2n+1}\left(\left(-q;q^2\right)_n^2-\left(-q^2;q^4\right)_n\right)\in \mathbb{Z}\left[q^2\right].
	\]
	Thus, by \eqref{eq:gf}, the definition of $\nu$ (see above \Cref{thm:watson}), and \eqref{thm:fine}, it follows that, for $n$ odd,
	\begin{align*}
	{\rm ou}^*(n)\equiv \left[q^{n}\right]\left(q\nu\left(-q^2\right)\right)\equiv \left[q^{n}\right]\left(q^3\omega\left(q^4\right)+q\frac{\left(q^8;q^8\right)_\infty^3}{\left(q^4;q^4\right)_\infty^2}\right)\pmod{4}.
	\end{align*}
	If $n\equiv 1\pmod{4}$, comparing coefficients of $q^{4m+1}$ (with $m\in \mathbb{N}_0$) on both sides, by the first identity of \Cref{assistlem0}, we get
	\begin{equation}\label{eqn3}
	{\rm ou^*}(n)\equiv c_1(m)\equiv c_\omega(4m+1)=c_{\omega}(n)\pmod{4}.
	\end{equation}
	
	Similarly, if $n\equiv 3\pmod{4}$, comparing coefficients of $q^{4m-1}$ (with $m\in \mathbb{N}$) on both sides, by the fourth identity of \Cref{assistlem0}, we obtain
	\[
	{\rm ou}^*(n)\equiv (-1)^{m} c_{\omega}(n)\pmod{4}.
	\]
	Combining this with \eqref{eqn3}, we conclude the proof.\qedhere 
\end{proof}

Next, we show the parity of $c_2(n)$.
\begin{lemma}\label{lem:h}
	For $n\in \mathbb{N}_0$, we have
	\[
	c_2(n)\ \text{is odd}\Leftrightarrow 24n+5\in S^{[1]}_2.
	\]
\end{lemma}
\begin{proof}
By \Cref{lem2}, we have
\[
(q)_\infty\left(q^4;q^4\right)_\infty=\sum_{r,s\in\Z}(-1)^{r+s}q^{\frac{r(3r-1)}2+2s(3s-1)}.
\]
Thus, by definition of $c_2(n)$ (see above \Cref{assistlem0}), 
\begin{equation}\label{eqn11}
c_2(n)\equiv\left|X_4(n)\right|\pmod2\ \ \text{with}\ \ X_4(n):=\left\{(r,s)\in\mathbb{Z}^2: \frac{r(3r-1)}2+2s(3s-1)=n\right\}.
\end{equation}
Note that
\[
X_4(n)=\left\{(r,s)\in\mathbb{Z}^2:r\equiv s\equiv 5\pmod{6}, r^2+4s^2=24n+5\right\}.
\]
Setting $X_5(n):=\{(r,s)\in\mathbb{Z}^2: r^2+4s^2=24n+5\}$, define a map $f_3:X_5(n)\rightarrow X_4(n)$ by $f_3(r,s):=(-(\frac r3)r,-(\frac s3)s)$. We observe that $f_3\mid_{X_4(n)}={\rm Id}$ and $f_3$ is surjective. Moreover, for $(r,s)\in X_4(n)$, $f^{-1}_3((r,s))\in \{(\pm r,\pm s)\}$, i.e., every fibre of $f_3$ has exactly $4$ distinct elements and thus, by \eqref{eqn11},
\begin{equation}\label{eqn12}
c_2(n)\equiv \left|X_4(n)\right|=\frac 14\left|X_5(n)\right|\pmod{2}.
\end{equation}

Next, define
\begin{align*}
X_6(n)&:=\left\{(r,s)\in\mathbb{Z}^2: r^2+s^2=24n+5\right\},\quad X^{[1]}_6(n):=\left\{(r,s)\in X_6: s\ \text{even}\right\},\\
  X^{[2]}_6(n)&:=\left\{(r,s)\in X_6: r\ \text{even}\right\}.
\end{align*}
Clearly, $|X_6(n)|=r_2(24n+5)=|X^{[1]}_6(n)|+|X^{[2]}_6(n)|$ because $X^{[1]}_6(n)\cap X^{[2]}_6(n)=\emptyset$. Define $f_4:X^{[2]}_6(n)\mapsto X^{[1]}_6(n)$ by $f_4((r,s)):=((s,r))$ and clearly, $f_4$ is a bijection. Therefore, 
\[
\left|X_6(n)\right|=r_2(24n+5)=2\left|X^{[1]}_6(n)\right|.
\]
Moreover, define $f_5:X_5(n)\rightarrow X^{[1]}_6(n)$ by $f_5((r,s)):=(r,2s)$ and noting the fact that $f^{-1}_5((r,s))=(r,\frac s2)$ is the inverse map, $f_5$ is a bijection. Therefore, by \eqref{eqn12} and \Cref{lem5}, we have
\begin{equation}\label{eqn13}
c_2(n)\equiv \frac{r_2\left(24n+5\right)}{8}=\frac{1}{2}\sum_{d\mid 24n+5}\chi_{-4}(d)\pmod{2}.
\end{equation}
Write\footnote{Here and throughout, for primes $\ell$, $\nu_{\ell}(n)$ denotes the $\ell$-adic valuation of $n$.} $24n+5=\prod_{\ell\mid 24n+5}\ell^{\nu_{\ell}(24n+5)}$ and using the fact that $\chi_{-4}$ is completely multiplicative, it is not hard to see that,
\begin{equation}\label{eqn14}
\frac{1}{2}\sum_{d\mid 24n+5}\chi_{-4}(d)=\frac 12\prod_{\ell\mid 24n+5}\sum_{j=0}^{\nu_\ell\left(24n+5\right)}\chi_{-4}(\ell)^j
\end{equation}
and 
\begin{equation}\label{eqn15}
\sum_{j=0}^{\nu_{\ell}\left(24n+5\right)}\chi_{-4}(\ell)^j=\begin{cases}\nu_\ell\left(24n+5\right)+1&\ell\equiv1\pmod4,\\ 1&\ell\equiv3\pmod 4\ \text{with}\ \nu_\ell\left(24n+5\right)\ \text{even},\\
0&\ell\equiv3\pmod 4\ \text{with}\ \nu_\ell\left(24n+5\right)\ \text{odd}.\end{cases}
\end{equation}

Now, assume that $24n+5\in S^{[1]}_2$. By definition of $ S^{[1]}_2$, $p\equiv 1\pmod{4}$, and consequently, by \eqref{eqn15}, we have
\[
\sum_{j=0}^{\nu_{p}\left(24n+5\right)}\chi_{-4}(p)^j\equiv 2\pmod{4},\qquad \sum_{\substack{j=0\\\ell\neq p}}^{\nu_{\ell}\left(24n+5\right)}\chi_{-4}(\ell)^j\equiv 1\pmod{2}.
\]
Plugging this into \eqref{eqn13} and then by \eqref{eqn13}, we conclude that $c_2(n)$ is odd.

Finally, we assume that $c_2(n)$ is odd. By \eqref{eqn13} and \eqref{eqn14}, this is equivalent to 
\begin{align}\label{eqn16}
&\sum_{d\mid 24n+5}\chi_{-4}(d)\equiv 2\pmod{4}\Leftrightarrow \!\!\!\prod_{\ell\mid 24n+5}\sum_{j=0}^{\nu_\ell\left(24n+5\right)}\chi_{-4}(\ell)^j\equiv 2\pmod{4}\nonumber\\
&\hspace{4 cm}\Leftrightarrow \!\!\!\prod_{\substack{\ell\mid 24n+5\\\ell\equiv 1\pmod{4}}}\!\!\!\sum_{j=0}^{\nu_\ell\left(24n+5\right)}\chi_{-4}(\ell)^j\!\!\prod_{\substack{\ell\mid 24n+5\\\ell\equiv 3\pmod{4}}}\!\!\!\sum_{j=0}^{\nu_\ell\left(24n+5\right)}\chi_{-4}(\ell)^j\equiv 2\pmod{4}.
\end{align}
Thus, by \eqref{eqn15}, for $\ell\equiv 3\pmod{4}$ and $\ell\mid 24n+5$, 
\begin{equation}\label{eqn17}
\nu_{\ell}\left(24n+5\right)\equiv 0\pmod{2}.
\end{equation}
Therefore, \eqref{eqn16} reduces to
\begin{align*}
&\prod_{\substack{\ell\mid 24n+5\\\ell\equiv 1\pmod{4}}}\!\!\!\sum_{j=0}^{\nu_\ell\left(24n+5\right)}\chi_{-4}(\ell)^j\equiv 2 \pmod{4}\Leftrightarrow \prod_{\substack{\ell\mid 24n+5\\\ell\equiv 1\pmod{4}}}\left(\nu_{\ell}\left(24n+5\right)+1\right)\equiv 2\pmod{4}\\
&\Leftrightarrow \nu_2\left(\prod_{\substack{\ell\mid 24n+5\\\ell\equiv 1\pmod{4}}}\left(\nu_{\ell}\left(24n+5\right)+1\right)\right)=1\Leftrightarrow \sum_{\substack{\ell\mid 24n+5\\\ell\equiv 1\pmod{4}}}\nu_2\left(\nu_{\ell}\left(24n+5\right)+1\right)=1.
\end{align*}
Since $\nu_2(n)\ge 0$ for $n\in \mathbb{N}$, this directly implies that there exists a unique prime, say $p\equiv1\pmod{4}$ with $p\mid 24n+5$ and 
\[
\nu_{p}\left(24n+5\right)+1\equiv 2\pmod{4}\Leftrightarrow \nu_p\left(24n+5\right)\equiv 1 \pmod{4}.
\]
Combining this and \eqref{eqn17}, $24n+5\in S^{[1]}_2$. This concludes the proof.\qedhere 
\end{proof}

Finally, we show congruences for $c_\omega(n)$ modulo $4$.

\begin{lemma}\label{assistlem1}
	Let $n\in \mathbb{N}$. Then, we have,
	\[
	c_\omega(n)\equiv 
	2\pmod4\Leftrightarrow\ 3n-1\in S_2.
	\]
\end{lemma}
\begin{proof}
First, we consider the case $n\equiv 0\pmod{8}$. Then, with $n=8m$ ($m\in \mathbb{N}_0$), $3n-1=24m-1$ is odd and by \Cref{assistlem0},  
\[c_\omega(n)\equiv2c_{\psi}(m)\pmod 4.\]
Thus, to prove the claim, it suffices to show that
\[
c_{\psi}(m)\ \text{is odd}\Leftrightarrow 24m-1\in S_2.
\] 
This follows from \Cref{thm:d}.

Now, we assume that $n\equiv1\pmod 4$. With $n=4m+1$ ($m\in\mathbb N_0$) we have $3n-1=2(6m+1)$, and by
\Cref{assistlem0}, it suffices to prove that
\[
c_1(m)\equiv 2\pmod 4 \Leftrightarrow 6m+1\in S_2^{[1]}.
\]
This holds by \Cref{L:CC}.

Next, we consider $n\equiv 2\pmod{8}$. Then, with $n=8m+2$ ($m\in \mathbb{N}_0$), $3n-1=24m+5$ is odd and by \Cref{assistlem0},  
\[c_\omega(n)\equiv2c_{2}(m)\pmod 4.\]
Thus, to prove the claim, it suffices to show that
\[
c_{2}(m)\ \text{is odd}\Leftrightarrow 24m+5\in S_2.
\] 
This follows from \Cref{lem:h}.

The case $n\equiv 4,6\pmod{8}$ is trivial because of \Cref{thm:apsy} and $3n-1\notin S_2$. Assuming on the contrary that $3n-1\in S_2$. This directly implies that $p\equiv r\pmod{8}$ with $r\in \{3,1\}$, respectively. Thus, $(\frac{-2}{p})=1$ and this contradicts the definition of $S_2$. 

Finally, we consider the remaining case $n\equiv 3\pmod{4}$. Write $n=4m-1$ with $m\in\mathbb N$. Note that, for the base case $n=3$, the claim holds trivially. Moreover, $m<n$ (as
$n-m=3m-1\ge2$), so by the induction hypothesis
\begin{equation}\label{E:IH}
c_\omega(m)\equiv 2\pmod 4\Leftrightarrow 3m-1\in S_2 .
\end{equation}
By the fourth congruence of \Cref{assistlem0}, 
\[
c_\omega(n)\equiv(-1)^mc_\omega(m)\pmod 4.
\]
Thus, we obtain
\begin{equation}\label{E:red1}
c_\omega(n)\equiv2\pmod 4\Leftrightarrow c_\omega(m)\equiv2\pmod 4 .
\end{equation}
Moreover, $3n-1=12m-4=4(3m-1)$. For $k\in\mathbb N$, it is straightforward to see that
\begin{equation}\label{E:red2}
4k\in S_2\iff k\in S_2.
\end{equation}
Combining this, \eqref{E:red1}, and \eqref{E:IH}, we conclude
\[
c_\omega(n)\equiv2\pmod 4\Leftrightarrow c_\omega(m)\equiv2\pmod 4\Leftrightarrow 3m-1\in S_2\Leftrightarrow 4(3m-1)=3n-1\in S_2.
\] 
This concludes the proof.\qedhere 
\end{proof}

\begin{proof}[Proof of \Cref{thm2}]
The first case trivially follows from \Cref{thm:BLparity}.

Now, we note that $S_1\cap S_2=\emptyset$ because for $m\in S_1$, $\nu_p(m)$ is even, whereas for $m\in S_2$, $\nu_p(m)$ is odd. Thus, to conclude the proof, it remains to show that, for $n\in \mathbb{N}$ odd, 
\begin{equation*}
{\rm ou}^*(n)\equiv 2 \pmod{4} \Leftrightarrow 3n-1\in S_2.
\end{equation*}
By \Cref{assistlem2} and \Cref{assistlem1}, we get
\[
{\rm ou}^*(n)\equiv 2 \pmod{4}\Leftrightarrow  c_{\omega}(n)\equiv 2 \pmod{4}\Leftrightarrow 3n-1\in S_2.\qedhere 
\]
\end{proof}

\begin{proof}[Proof of \Cref{C:GD2}]
Let $n\in\mathbb N_0$, $N:=2M^2\ell^2n+R_{M,\ell,j}$, $B:=M\ell^{-\nu_\ell(M)}$ and $K:=3\ell n+3j+\ell$. Then
\[
3N-1=2\ell^{2\nu_\ell(M)+1}B^2K,\qquad \ell\nmid BK,\quad K\equiv3j\pmod\ell,\quad K\equiv\ell\pmod 3 .
\]
Now, we show that
\[
N\ \text{is odd},\quad 3N-1\notin S_1,\quad 3N-1\notin S_2.
\]
Since $\gcd(M,6)=1$ and $\ell\ge5$ is prime, $\gcd(M\ell,6)=1$. Thus $\frac{2M^2\ell^2+1}3$ is odd and so, $N$ is odd. 

Next, note that $3N-1\notin S_1$ because $\nu_\ell(3N-1)$ is odd. 

Now, suppose that $3N-1\in S_2$. Write $3N-1=\varepsilon p^{4a+1}m^2$ with $\varepsilon\in \{1,2\}$.  Since, by definition of $S_2$,
\[
\nu_p\left(3N-1\right)=4a+1,\quad \nu_q(3N-1)\ \text{is even}\ \text{for all primes}\ q\neq p.
\] 
Moreover, as $\nu_\ell(3N-1)=2\nu_\ell(M)+1$ is odd, we get $p=\ell$. Thus, $\nu_\ell(M)=2a$ which contradicts the assumption that $\nu_{\ell}(M)$ is odd. Since $p=\ell$ and by definition of $S_2$, $(\frac{-2}\ell)=-1$, we have $\ell\equiv5,7\pmod 8$. In particular, condition (1) of \Cref{C:GD2} fails to hold as well. Note that
\[
2\ell^{2\nu_\ell(M)+1}B^2K=\varepsilon p^{4a+1}m^2\Rightarrow \varepsilon m^2=2B^2K\ \ \text{with}\ \varepsilon\in \{1,2\}.
\]
Write $K=\frac{\varepsilon s^2}{2}$ with $s:=\frac{m}{B}$ and $3\nmid s, \ell\nmid s$ because $K$ is not divisible by $3$ and $\ell$. For $K=s^2$,
\[
\ell\equiv K=s^2\equiv1\pmod 3\ \ \text{and}\ \ \left(\frac{-2}\ell\right)=-1\Rightarrow \ell\equiv5,23\pmod{24}
\]
and $3j\equiv K=s^2\pmod\ell$ implies that $(\frac{3j}\ell)=1$. Thus condition (3) of \Cref{C:GD2} are not satisfied. Similarly, for $K=2s^2$, $\ell\equiv K=2s^2\equiv2\pmod 3$ and $(\frac{-2}\ell)=-1$ imply that $\ell\equiv5,23\pmod{24}$ and $6j\equiv2K=(2s)^2\pmod\ell$ imply that $(\frac{6j}\ell)=1$. It is clear that, for $\ell\equiv5\pmod{24}$, 
$(\frac2\ell)=(\frac3\ell)=-1$, and for $\ell\equiv23\pmod{24}$, $(\frac2\ell)=(\frac3\ell)=1$.
So $(\frac6\ell)=1$ and $(\frac j\ell)=1$. Hence condition (2) of \Cref{C:GD2} fails to hold. Thus, for $3N-1\in S_2$, all the conditions in \Cref{C:GD2} fail to satisfy and thus contradicts the assumption. Therefore, $3N-1\notin S_2$ and by \Cref{thm2}, we conclude the proof.

Finally, assume that $\nu_\ell(M)=2a$ is even and that the conditions (1)--(3) do not hold. If $\ell\equiv7,13\pmod{24}$ and
$(\frac{3j}\ell)=1$, we choose $s\in\mathbb N$ with $s^2\equiv3j\pmod\ell$, $s\equiv1\pmod 3$. Moreover, we set
$n:=\frac{s^2-3j-\ell}{3\ell}$. Then $K=s^2$ and $3N-1=2\ell^{4a+1}(Bs)^2$. 

Similarly, if $\ell\equiv5,23\pmod{24}$ and
$(\frac j\ell)=1$, then $(\frac{6j}\ell)=1$ as $\left(\frac6\ell\right)=1$, so we choose $s\in\mathbb N$ with
$2s^2\equiv3j\pmod\ell$, $s\equiv1\pmod 3$. Letting $n:=\frac{2s^2-3j-\ell}{3\ell}$, $K=2s^2$ and
$3N-1=\ell^{4a+1}(2Bs)^2$. In both cases $n\in\mathbb Z$, since the numerator is divisible by $\ell$ and by $3$ (as $3\nmid$), and $n\in \mathbb{N}_0$ for $s$ sufficiently large. Since $\ell\nmid Bs$ and $(\frac{-2}\ell)=-1$, we get $3N-1\in S_2$ and so, by \Cref{thm2}, $\mathrm{ou}^*(N)\equiv2\pmod 4$. Since $s$ can be chosen arbitrarily large, this holds for infinitely many $n\in\mathbb{N}_0$.\qedhere   
\end{proof}

\section{Open questions}\label{sec4}
We conclude this paper by proposing the following questions and conjectures related to ${\rm ou}(n)$ and ${\rm ou}^*(n)$.
\begin{conjecture}\label{con1}
	Let $\a\in \mathbb{N}_{\ge 3}$, $A_{\alpha}, B_{\alpha}\in \mathbb{N}$.	Determine (if exists) arithmetic progression $A_{\alpha}n+B_{\alpha}$ such that, for  all $n\in \mathbb{N}_0$, 
	\[{\rm ou}\left(A_{\alpha}n+B_{\alpha}\right)\equiv 0\pmod{2^{\alpha}}.\]
\end{conjecture}
\begin{conjecture}\label{conj2}
Let $\alpha\in \mathbb{N}_{\ge 3}$. Then there does not exist $C_{\alpha}, D_{\alpha}\in \mathbb{N}$, such that for all $n\in \mathbb{N}_0$, 
\[{\rm ou}^*\left(C_{\alpha}n+D_{\alpha}\right)\equiv 0\pmod{2^{\alpha}}.\]
\end{conjecture}

\end{document}